\documentclass[reqno]{amsart}
\usepackage{mathrsfs}
\usepackage{color}
\usepackage{amsmath}
\usepackage{amsfonts}
\usepackage{amssymb}
\usepackage{graphicx}%
\usepackage{hyperref}

 \newtheorem{Theorem}{Theorem}[section]
 \newtheorem{Corollary}{Corollary}[section]
 \newtheorem{Lemma}{Lemma}[section]
 \newtheorem{Proposition}{Proposition}[section]

 \newtheorem{Conjecture}{Conjecture}[section]

 \newtheorem{Remark}{Remark}[section]

 \numberwithin{equation}{section}

\begin{document}

\title[A proof of Saitoh's conjecture]
 {A proof of Saitoh's conjecture for conjugate Hardy $H^{2}$ kernels}

\author{Qi'an Guan}
\address{Qi'an Guan: School of Mathematical Sciences,
Peking University, Beijing, 100871, China.}
\email{guanqian@math.pku.edu.cn}

\thanks{}

\subjclass[2010]{}


\date{\today}

\dedicatory{}

\commby{}


\begin{abstract}
In this article,
we obtain a strict inequality between the conjugate Hardy $H^{2}$ kernels and the Bergman kernels on planar regular regions with $n>1$ boundary components,
which is a conjecture of Saitoh.
\end{abstract}

\maketitle

\section{Introduction}

Let $D$ be a planar regular region with $n$ boundary components which are analytic Jordan curves (see \cite{Saitoh88,yamada98}).

As in \cite{Saitoh88}, $H_{2}^{(c)}(D)$ denotes the analytic Hardy class on $D$ defined as the set of
all analytic functions $f(z)$ on $D$ such that the subharmonic functions $|f(z)|^{2}$ have harmonic majorants $U(z)$,
i.e. $|f(z)|^{2}\leq U(z)$ on $D$.

As in \cite{Saitoh88}, $\hat{R}_{t}(z,\bar{w})$ denotes the conjugate Hardy $H^{2}$ kernel on $D$
if
\begin{equation}
f(w)=\frac{1}{2\pi}\int_{\partial D}f(z)\overline{\hat{R}_{t}(z,\bar{w})}(\frac{\partial G(z,t)}{\partial \nu_{z}})^{-1}d|z|
\end{equation}
holds for any holomorphic function $f\in H_{2}^{(c)}(D)$ which satisfies
$$\int_{\partial D}|f(z)|^{2}(\frac{\partial G(z,t)}{\partial \nu_{z}})^{-1}d|z|<+\infty,$$
where $f(z)$ means Fatou's nontangential boundary value, and  $\frac{\partial}{\partial \nu_{z}}$ denotes the derivative along the outer normal unit vector $\nu_{z}$. It is well-known that $\frac{\partial G(z,t)}{\partial \nu_{z}}$ is positive continuous on $\partial D$
because of the analyticity of the boundary (see \cite{Saitoh88}).

When $t=w$, $\hat{R}(z,\bar{w})$ denotes $\hat{R}_{w}(z,\bar{w})$ for simplicity.
When $z=w$, $\hat{R}(z)$ denotes $\hat{R}(z,\bar{z})$ for simplicity.

Let $B(z,\bar{w})$ be the Bergman kernel on $D$.
When $z=w$, $B(z)$ denotes $B(z,\bar{z})$ for simplicity.

In \cite{yamada98} (see also \cite{Saitoh88} and \cite{yamada99}), the following so-called Saitoh's conjecture was posed (backgrounds and related results could be referred to Hejhal's paper \cite{Hejhal} and Fay's book \cite{fay}).

\begin{Conjecture}\label{conj:saitoh}(Saitoh's Conjecture)
If $n>1$, then $\hat{R}(z)>\pi B(z)$.
\end{Conjecture}

In the present article, we give a proof of the above Conjecture.
\begin{Theorem}
\label{thm:sarioh_conj}
Conjecture \ref{conj:saitoh} holds.
\end{Theorem}

One of the ingredients of the present article is using the concavity of minimal $L^{2}$ integrations in \cite{guan-effect}.

\section{Preparations}

In the present section, we recall some known results and present some preparations,
which will be used in the proof of Theorem \ref{thm:sarioh_conj}.

\subsection{The concavity of minimal $L^{2}$ integrations}

Let $G(z,w)$ be the Green's function on $D\times D$ such that $G(z,w)-\log|z-w|$ is analytic on $D\times D$.

Let $g_{w}(t)$ be the minimal $L^{2}$ integration of the holomorphic functions $f$ on $\{2G(\cdot,w)<-t\}$ satisfying $f(w)=1$.
In \cite{guan-effect},
we establish the following concavity of the $g_{w}(-\log r)$.
\begin{Proposition}
\label{prop:logconcave}(Proposition 4.1 in \cite{guan-effect})
$g_{w}(-\log r)$ is concave with respect to $r\in(0,1]$.
\end{Proposition}

Note that $\lim_{r\to0+0}g_{w}(-\log r)=0$, then Proposition \ref{prop:logconcave} implies that

\begin{Corollary}
\label{coro:concave_domain}
The inequality
\begin{equation}
\label{equ:concave_domain}
\lim_{r\to 1-0}\frac{g_{w}(-\log r)-g_{w}(0)}{r-1}\leq g_{w}(0)\leq \lim_{r\to 0+0}\frac{g_{w}(-\log r)}{r}
\end{equation}
holds for any $r\in(0,1)$,
where $\lim_{r\to 0+0}\frac{g_{w}(-\log r)}{r}$ might be $+\infty$.
Moreover, the following three statements are equivalent

$(1)$ $\lim_{r\to 1-0}\frac{g_{w}(-\log r)-g_{w}(0)}{r-1}=g_{w}(0)$;

$(2)$ $\lim_{r\to 0+0}\frac{g_{w}(-\log r)}{r}=g_{w}(0)$;

$(3)$ $g_{w}(-\log r)=rg(0)$ for any $r\in(0,1]$.
\end{Corollary}

\subsection{Green's function and Bergman kernel}
Note that there exists a local coordinate $z'$ on a neighborhood $U_{0}$ of $z_{0}\in\partial D$ such that
$\partial D|_{U_{0}}=\{\Im z'=0\}$, which implies the following well-known Lemma.

\begin{Lemma}
\label{lem:green}
The Green's function $G(z,w)$ has an analytic extension on $(U\times V)\setminus \{z=w\}$,
where $U$ is a neighborhood of $\bar{D}$ and $V\subset\subset D$.
\end{Lemma}

Note that the Bergman kernel $B(z,\bar{w})$ on $D\times D$ equals $\frac{\partial}{\partial z}\frac{\partial}{\partial\bar{w}}G(z,w)$ (see \cite{Bergman}),
then it follows from Lemma \ref{lem:green} that
$B(\cdot,\bar{w})$ is smooth on a neighbohood of $\bar{D}$ for any given $w\in D$.
Note that $\frac{B(\cdot,\bar{w})}{B(w,\bar{w})}$ is the (unique) holomorphic function satisfying $\int_{D}|\frac{B(\cdot,\bar{w})}{B(w,\bar{w})}|^{2}=g_{w}(0)$
and $\frac{B(\cdot,\bar{w})}{B(w,\bar{w})}(w)=1$ (see \cite{Bergman}),
then it follows that
\begin{Remark}
\label{rem:bergman_smooth}
There exists a (unique) holomorphic function $f$ $(=\frac{B(\cdot,\bar{w})}{B(w,\bar{w})})$,
which is smooth on a neighhorhood of $\bar{D}$ such that
$f(w)=1$ and $\int_{D}|f|^{2}=g_{w}(0)$.
\end{Remark}

\subsection{A solution of a conjecture of Suita}

We recall the following solution of a conjecture posed by Suita \cite{suita}.
\begin{Theorem}\label{thm:suita}(\cite{guan-zhou13ap})
$(c_{\beta}(z_{0}))^{2}=\pi B(z_{0})$ holds for some $z_{0}\in D$
if and only if $D$ conformally equivalent to the unit disc i.e. $(n=1)$,
where $c_{\beta}(z_{0})=\lim_{z\to z_{0}}\exp(G(z,z_{0})-\log|z-z_{0}|)$.
\end{Theorem}

Note that $2G(z,z_{0})-2\log|z-z_{0}|$ is harmonic on $D$ (continuous near $z_{0}$),
then it follows that
\begin{Lemma}
\label{lem:log_capacity}
$\frac{(c_{\beta}(z_{0}))^{2}}{\pi}=\lim_{r\to 0+0}\frac{r}{g_{z_{0}}(-\log {r})}$.
\end{Lemma}

\begin{proof}
Note that $c_{\beta}(z_{0})=\lim_{z\to z_{0}}\exp(G(z,z_{0})-\log|z-z_{0}|)$ implies that
for any $\varepsilon>0$, then there exists a neighborhood $U_{0}$ of $z_{0}$ such that
\begin{equation}
\label{equ:log_cap}
|G(z,z_{0})-\log (c_{\beta}(z_{0})|z-z_{0}|)|<\varepsilon.
\end{equation}
As $G(z,z_{0})$ and $\log (c_{\beta}(z_{0})|z-z_{0}|)$ both go to $-\infty$ ($z\to z_{0}$),
then there exists $\delta_{0}>0$, such that $\{G(z,z_{0})<\frac{1}{2}\log r\}\subset U_{0}$
and $\{\log (c_{\beta}(z_{0})|z-z_{0}|)<\frac{1}{2}\log r\}\subset U_{0}$ for any $r\in(0,\delta_{0})$.
It follows from inequality \ref{equ:log_cap} that for any $r\in(0,\delta_{0}e^{-2\varepsilon})$,
\begin{equation}
\label{equ:subset}
\begin{split}
\{\log (c_{\beta}(z_{0})|z-z_{0}|)+\varepsilon<\frac{1}{2}\log r\}
&\subset\{G(z,z_{0})<\frac{1}{2}\log r\}
\\&\subset\{\log (c_{\beta}(z_{0})|z-z_{0}|)-\varepsilon<\frac{1}{2}\log r\}\subset U_{0}
\end{split}
\end{equation}
which implies
\begin{equation}
\label{equ:measure}
\begin{split}
\mu\{\log (c_{\beta}(z_{0})|z-z_{0}|)+\varepsilon<\frac{1}{2}\log r\}
&\leq g_{z_{0}}(-\log {r})
\\&\leq\mu\{\log (c_{\beta}(z_{0})|z-z_{0}|)-\varepsilon<\frac{1}{2}\log r\},
\end{split}
\end{equation}
i.e.
\begin{equation}
c^{-2}_{\beta}(z_{0})\pi r e^{-2\varepsilon}\leq g_{z_{0}}(-\log {r})\leq c^{-2}_{\beta}(z_{0})\pi r e^{2\varepsilon},
\end{equation}
where $\mu$ is the Lebesgue measure on $\mathbb{C}$.
Then Lemma \ref{lem:log_capacity} has been proved by the arbitrariness of $\varepsilon>0$.
\end{proof}
Note that
\begin{equation}
\label{equ:Bergman}
B(z_{0})=\frac{1}{g_{z_{0}}(0)},
\end{equation}
then it follows from Theorem \ref{thm:suita} and Lemma \ref{lem:log_capacity} that

\begin{Remark}
\label{rem:suita}
Statement (2) in Corollary \ref{coro:concave_domain} holds if and only if $D$ is the unit disc.
\end{Remark}

\subsection{Conjugate analytic Hardy space on $D$}

\begin{Lemma}
\label{lem:conjugate}
For any given $w_{0}\in D$,
and holomorphic function $f$ on $D$ which is continuous on $\bar{D}$,
\begin{equation}
\label{equ:concave_boundary}
\lim_{r\to 1-0}\frac{\int_{\{e^{2G(z,w_{0})}\geq r\}}|f(z)|^{2}}{1-r}=\int_{\partial D}|f(z)|^{2}(\frac{\partial 2G(z,w_{0})}{\partial \nu_{z}})^{-1}d|z|
\end{equation}
holds,
where $\frac{\partial}{\partial \nu_{z}}$ is the derivative along the outer normal unit vector $\nu_{z}$.
\end{Lemma}

\begin{proof}

As $\frac{\partial}{\partial \nu_{z}}G(z,w_{0})$ is positive on $\partial D$,
it is clear that $\frac{\partial}{\partial y}G(z_{b},w_{0})\neq 0$ or $\frac{\partial}{\partial x}G(z_{b},w_{0})\neq 0$,
where $z_{b}\in\partial D$.
Then there exists a neighborhood $U_{b}$ of $z_{b}$
with coordinates $(u,v)=(x,2G(x+\sqrt{-1}y,w_{0}))$ or $(2G(x+\sqrt{-1}y,w_{0}),y)$ on $U_{b}$.
Note that $\partial D$ is compact,
then there exist finite $U_{b}$ covering $\partial D$.
It is clear that one can choose finite unitary decomposition $\{\rho_{\lambda}\}_{\lambda}$ such that $\sum\rho_{\lambda}=1$ near $\partial D$,
and for any $\lambda$, $Supp(\rho_{\lambda})\subset U_{b}$ for some $z_{b}$.

Without losing of generality,
we assume that $\frac{\partial}{\partial y}G(z_{b},w_{0})\neq 0$,
where $z_{b}\in\partial D$.
Then there exists a neighborhood $U_{b}$ of $z_{b}$
with coordinates $(u,v)=(x,2G(x+\sqrt{-1}y,w_{0}))$,
where $u\in (a_{1},a_{2}),v\in (\log r_{b},-\log r_{b})$ and $r_{b}\in(0,1)$.

It suffices to consider that $|f|^{2}\rho$ instead of $|f|^{2}$ in equality \ref{equ:concave_boundary},
where $Supp\{\rho\}\subset\subset U_{b}$ and $\rho$ is smooth.
It is clear that $\frac{\partial u}{\partial x}=1$, $\frac{\partial u}{\partial y}=0$,
$\frac{\partial v}{\partial x}=\frac{\partial}{\partial x}2G(z,w_{0})$,
and $\frac{\partial v}{\partial y}=\frac{\partial}{\partial y}2G(z,w_{0})$,
which implies that
$\frac{\partial x}{\partial u}=1$, $\frac{\partial y}{\partial u}=-\frac{\frac{\partial}{\partial x}2G(z,w_{0})}{\frac{\partial}{\partial y}2G(z,w_{0})}$,
$\frac{\partial x}{\partial v}=0$,
and $\frac{\partial y}{\partial v}=(\frac{\partial}{\partial y}2G(z,w_{0}))^{-1}$.
It is clear that equalities
$$\nu_{z}=\frac{(\frac{\partial}{\partial x}2G(z,w_{0}),\frac{\partial}{\partial y}2G(z,w_{0}))}{((\frac{\partial}{\partial x}2G(z,w_{0}))^{2}+(\frac{\partial}{\partial y}2G(z,w_{0}))^{2})^{1/2}}$$
and
$$\frac{\partial}{\partial \nu_{z}}2G(z,w_{0})=((\frac{\partial}{\partial x}2G(z,w_{0}))^{2}+(\frac{\partial}{\partial y}2G(z,w_{0}))^{2})^{1/2}$$
hold,
which implies
\begin{equation}
\label{equ:integ_curve}
((\frac{\partial x}{\partial u})^{2}+(\frac{\partial y}{\partial u})^{2})^{1/2}=\frac{\frac{\partial}{\partial \nu_{z}}2G(z,w_{0})}{|\frac{\partial}{\partial y}2G(z,w_{0})|}
\end{equation}

Replacing the integral variables,
one can obtain that
\begin{equation}
\begin{split}
&\int_{\{e^{2G(z,w)}\geq r\}}|f(z)|^{2}\rho
\\=&\int_{\{a_{1}< u< a_{2},\log r\leq v\leq0\}}|f(z(u,v))|^{2}\rho(z(u,v))(|\frac{\partial}{\partial y}2G(z(u,v),w_{0})|)^{-1},
\end{split}
\end{equation}
which implies that
\begin{equation}
\begin{split}
&\lim_{r\to 1-0}\frac{\int_{\{e^{2G(z,w)}\geq r\}}|f(z)|^{2}\rho}{1-r}
\\=&\lim_{r\to 1-0}\frac{\int_{\{a_{1}< u< a_{2},\log r\leq v\leq0\}}|f(z(u,v))|^{2}\rho(z(u,v))(|\frac{\partial}{\partial y}2G(z(u,v),w_{0})|)^{-1}}{1-r}
\\=&\lim_{r\to 1-0}\frac{\int_{\{\{a_{1}< u< a_{2},\log r\leq v\leq0\}}|f(z(u,v))|^{2}\rho(z(u,v))(|\frac{\partial}{\partial y}2G(z(u,v),w_{0})|)^{-1}}{-\log r}
\\=&\int_{\{a_{1}< u< a_{2}\}}|f(z(u,0))|^{2}\rho(z(u,0))(|\frac{\partial}{\partial y}2G(z(u,0),w_{0})|)^{-1}du
\\=&\int_{\partial D}|f(z)|^{2}\rho(z)(\frac{\partial 2G(z,w_{0})}{\partial \nu_{z}})^{-1}d|z|,
\end{split}
\end{equation}
where the last equality is from equality \ref{equ:integ_curve}.
Then Lemma \ref{lem:conjugate} has been proved.
\end{proof}

\section{Proof of Theorem \ref{thm:sarioh_conj}}

We prove Theorem \ref{thm:sarioh_conj} by two steps:
firstly we prove that $"\geq"$ holds, secondly we prove that $"="$ doesn't hold.

\textbf{Step 1.}
Let $f(z)=\frac{B(\cdot,\bar{w})}{B(w,\bar{w})}$, which implies that
\begin{equation}
\label{equ:mini}
\int_{D}|f|^{2}=g_{w}(0).
\end{equation}
It follows from Remark \ref{rem:bergman_smooth} that
$f$ is continuous on $\bar{D}$,
which implies that
$1=f(w)=\frac{1}{2\pi}\int_{\partial D}f(z)\overline{\hat{R}(z,\bar{w})}(\frac{\partial G(z,w)}{\partial \nu_{z}})^{-1}d|z|$.
By Cauchy-Schwartz Lemma, it follows that
\begin{equation}
\begin{split}
1\leq &\frac{1}{(2\pi)^{2}}(\int_{\partial D}|f(z)|^{2}(\frac{\partial G(z,w)}{\partial \nu_{z}})^{-1}d|z|)
\\&\times(\int_{\partial D}|\hat{R}(z,\bar{w})|^{2}(\frac{\partial G(z,w)}{\partial \nu_{z}})^{-1}d|z|).
\end{split}
\end{equation}

As $f$ is continuous on $\bar{D}$, it follows from Lemma \ref{lem:conjugate} that
\begin{equation}
\label{equ:conjugate}
\begin{split}
\lim_{r\to 1-0}\frac{1-r}{\int_{\{e^{2G(z,w)}\geq r\}}|f(z)|^{2}}
=&(\int_{\partial D}|f(z)|^{2}(\frac{\partial 2G(z,w)}{\partial \nu_{z}})^{-1}d|z|)^{-1}
\\\leq& \frac{1}{2\pi^{2}}(\int_{\partial D}|\hat{R}(z,\bar{w})|^{2}(\frac{\partial G(z,w)}{\partial \nu_{z}})^{-1}d|z|)
\\&=\frac{1}{2\pi^{2}}(\int_{\partial D}\hat{R}(z,\bar{w})\overline{\hat{R}(z,\bar{w})}(\frac{\partial G(z,w)}{\partial \nu_{z}})^{-1}d|z|)
\\&=\frac{1}{\pi}\hat{R}(w,\bar{w})=\frac{1}{\pi}\hat{R}(w)
\end{split}
\end{equation}

Note that
$$g_{w}(-\log r)\leq \int_{\{e^{2G(z,w)}< r\}}|f(z)|^{2},$$
then it follows from equality \ref{equ:mini} that
$$g_{w}(0)-g_{w}(-\log r)\geq\int_{D}|f(z)|^{2}-\int_{\{e^{2G(z,w)}< r\}}|f(z)|^{2}=\int_{\{e^{2G(z,w)}\geq r\}}|f(z)|^{2},$$
which implies
\begin{equation}
\label{equ:boundary}
\lim_{r\to 1-0}\frac{r-1}{g_{w}(-\log r)-g_{w}(0)}\leq\lim_{r\to 1-0}\frac{1-r}{\int_{\{e^{2G(z,w)}\geq r\}}|f(z)|^{2}}.
\end{equation}

It follows from equality \ref{equ:Bergman} and inequality \ref{equ:concave_domain} \ref{equ:boundary} \ref{equ:conjugate} that
\begin{equation}
\label{equ:Sandwich}
\begin{split}
B(w)=\frac{1}{g_{w}(0)}&\leq\lim_{r\to 1-0}\frac{r-1}{g_{w}(-\log r)-g_{w}(0)}
\\&\leq\lim_{r\to 1-0}\frac{1-r}{\int_{\{e^{2G(z,w)}\geq r\}}|f(z)|^{2}}\leq\frac{1}{\pi}\hat{R}(w).
\end{split}
\end{equation}

Then we obtain that $"\geq"$ holds.

\
\textbf{Step 2.}
It suffices to prove $B(w)\neq\frac{1}{\pi}\hat{R}(w)$.
We prove by contradiction:
if not, then $B(w)=\frac{1}{\pi}\hat{R}(w)$ holds.
It follows from \ref{equ:Sandwich} that $\frac{1}{g_{w}(0)}=\lim_{r\to 1-0}\frac{r-1}{g_{w}(-\log r)-g_{w}(0)}$ (statement (1) of Corollary \ref{coro:concave_domain}).
Combining with Corollary \ref{coro:concave_domain} and Remark \ref{rem:suita}, we obtain that $n=1$ which contradicts $n>1$.

Then Theorem \ref{thm:sarioh_conj} has been proved.

\vspace{.1in} {\em Acknowledgements}.
The author would like to thank Professor Xiangyu Zhou for bringing me to Saitoh's conjecture when I was a postdoctor,
and Professor Fusheng Deng, Professor Takeo Ohsawa, Professor Saburou Saitoh for helpful discussions and encouragements.
The author would also like to
thank the hospitality of Beijing International Center for Mathematical Research.

The author was partially supported by NSFC-11522101 and NSFC-11431013.

\bibliographystyle{references}
\bibliography{xbib}

\end{document}